\documentclass[11pt,twoside]{article}

\usepackage[left=3cm, right=3cm, top=2.5cm]{geometry}
\usepackage{amssymb}
\usepackage{amsmath}
\usepackage{amsthm}
\usepackage{mathtools}
\usepackage{listings}
\usepackage{epsfig}
\usepackage{latexsym}
\usepackage{mathtools}
\usepackage{booktabs}
\usepackage{graphicx}
\usepackage{xcolor}
\usepackage{epsfig}
\usepackage{url}
\usepackage{amsfonts}

\usepackage{hyperref}
\usepackage{tipa}
\usepackage{dsfont}
\usepackage[mathscr]{euscript}
\usepackage{enumitem}
\usepackage{bm}
\usepackage[backend=bibtex,style=numeric,citestyle=numeric,giveninits=true,sortcites=true,maxbibnames=10]{biblatex}
\usepackage{dsfont}
\usepackage{subcaption}
\usepackage{appendix}
\usepackage{float}

\usepackage{filecontents}

\immediate\write18{bibtex \jobname}
\usepackage{color}
\definecolor{deepblue}{rgb}{0,0,0.5}
\definecolor{deepred}{rgb}{0.6,0,0}
\definecolor{deepgreen}{rgb}{0,0.5,0}

	\definecolor{DarkBlue}{rgb}{0.00,0.00,0.55}
	\definecolor{Black}{rgb}{0.00,0.00,0.00}
	\hypersetup{
		linkcolor = DarkBlue,
		anchorcolor = DarkBlue,
		citecolor = DarkBlue,
		filecolor = DarkBlue,
		colorlinks  = true,
		urlcolor = Black
	}
	
\graphicspath{ {./} }
\DeclareGraphicsExtensions{.eps,.pdf,.png,.jpg}  
	
\newtheorem{theorem}{Theorem}[section]
\newtheorem{lemma}[theorem]{Lemma}

\theoremstyle{definition}
\newtheorem{definition}{Definition}[section]

\theoremstyle{remark}

\newcommand{\TheTitle}{Complexity bounds on supermesh construction for quasi-uniform meshes}
\newcommand{\TheAuthors}{M.~Croci and P.~E.~Farrell}

\title{{\TheTitle}\thanks{\textbf{Funding:} This research is supported by EPSRC grants EP/R029423/1, and by the EPSRC Centre For Doctoral Training in Industrially Focused Mathematical Modelling (EP/L015803/1) in collaboration with Simula Research Laboratory.}}

\author{
  M. Croci${}^\ddagger$\thanks{Mathematical Institute, University of Oxford, Oxford, UK. (\textbf{\url{matteo.croci@maths.ox.ac.uk}}), (\textbf{\url{patrick.farrell@maths.ox.ac.uk}}).}
  \and
  P.~E.~Farrell\footnotemark[2]
}

\usepackage{amsopn}
\DeclareMathOperator{\R}{\mathbb{R}}

\DeclareMathOperator{\T}{\mathcal{T}}
\DeclareMathOperator{\diam}{\text{diam}}
\DeclareMathOperator{\B}{\mathcal{B}}

\DeclareMathOperator{\I}{\mathcal{I}}
\renewcommand{\P}{\mathscr{P}}

\usepackage{todonotes}
\definecolor{myblue}{RGB}{135, 206, 250}

\ifpdf
\hypersetup{
  pdftitle={\TheTitle},
  pdfauthor={\TheAuthors}
}
\fi

\begin{document}

\maketitle

\begin{abstract}
		Projecting fields between different meshes commonly arises in computational physics. This operation requires a supermesh construction and its computational cost is proportional to the number of cells of the supermesh $n$. Given any two quasi-uniform meshes of $n_A$ and $n_B$ cells respectively, we show under standard assumptions that $n$ is proportional to $n_A + n_B$. This result substantially improves on the best currently available upper bound on $n$ and is fundamental for the analysis of algorithms that use supermeshes.
\end{abstract}


\vspace{0.8cm}
\section{Introduction}

The problem of projecting fields between non-nested meshes frequently arises in computational physics and scientific computing.
For this operation, the key ingredient is a \emph{supermesh} construction, a mesh that is a common refinement of both input meshes, cf.~figure \ref{fig:supermesh}. Applications of supermeshes arise in adaptive remeshing, diagnostic computation, multimesh discretisations, cut finite element methods, and multilevel Monte Carlo, among others~\cite{maddison2012directional, aguerre2017conservative,PlantaEtAl2018,SchlottkeEtAl2019,Johansson2019multimesh,Burman2015CutFEM,croci2018efficient}. Several efficient algorithms for supermesh construction have been published (cf.~\cite{Farrell2009Supermesh, Farrell2011Supermesh, menon2011conservative, krause2016parallel}). In applications, complexity bounds for supermesh construction are essential for the analysis of the algorithm as a whole. Typically, the total supermeshing cost is proportional to the number of cells $n$ of the resulting supermesh, for which only crude estimates are currently available; cf.~\cite{Farrell2011Supermesh}. In this note, we substantially improve on the best currently available upper bound for $n$.

More specifically, we show that the number of cells of a supermesh between two quasi-uniform meshes is linear in the sum of the number of cells of the parent meshes. We prove the following result.

\newpage

\begin{figure}[t!]
	\centering
	\includegraphics[width=0.9\linewidth]{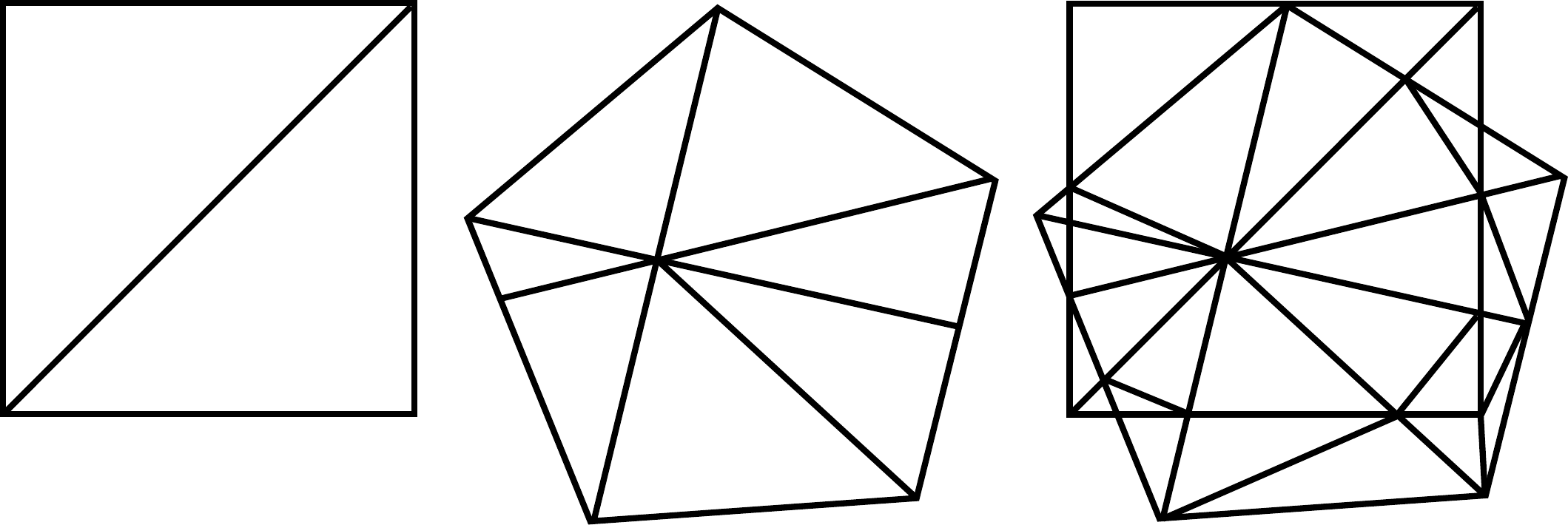}
	\centering
	\caption{\textit{An example of a supermesh construction with non-overlapping domains. The first two meshes on the left are the parent meshes and the mesh on the right is one of their supermeshes.}}
	\label{fig:supermesh}
	\vspace{-12pt}
\end{figure}
\begin{theorem}
	\label{th:quasi_uniform_supermeshing_main}
    Let $A_h$ and $B_h$ be two quasi-uniform tessellations of two bounded domains $A,B\in\R^d$. Let $n_A\geq1$ and $n_B\geq 1$ be the number of cells of $A_h$ and $B_h$ respectively and let $S_h$ be a supermesh of $A_h$ and $B_h$ constructed in such a way that each pair of intersecting cells is triangulated into a finite number of supermesh cells. Then there exists a constant $C(d,A,B) \geq 1/2$ independent of $n_A$ and $n_B$ such that the number $n$ of cells of $S_h$ is bounded above by
    \begin{align}
    n \leq C(d,A,B) (n_A + n_B).
    \end{align}
\end{theorem}
This result improves on the previously known bound, proven e.g.~in \cite{Farrell2011Supermesh}, which states that ${n \leq \tilde{C}(d,A,B) n_A n_B}$ for some positive constant $\tilde{C}(d,A,B)$. Note that theorem \ref{th:quasi_uniform_supermeshing_main} applies to any type of (quasi-uniform) tessellation with any cell shape.

\section{Preliminaries}

Before proceeding, we briefly recall the definitions of a supermesh and of a quasi-uniform tessellation.

\begin{definition}[Supermesh, \cite{Farrell2011Supermesh, Farrell2009Supermesh}]
	Let $A,B\subset\R^d$ be two domains and let $A_h$, $B_h$ be tessellations of $A$ and $B$ respectively. A supermesh $S_h$ of $A_h$ and $B_h$ is a common refinement of $A_h$ and $B_h$. More specifically, $S_h$ is a triangulation of $S = A \cup B$ such that:
	\begin{enumerate}[leftmargin=1cm]
		\item $\textnormal{vertices}(A_h) \cup \textnormal{vertices}(B_h) \subseteq \textnormal{vertices}(S_h)$,
		\item $|e_S \cap e| \in \{0, |e_S|\}$ for all cells $e_S\in S_h$, $e\in (A_h \cup B_h)$.
	\end{enumerate}
\end{definition}
Here we indicate with $|D|$ the measure of a domain $D\subset\R^d$. The first condition means that every parent mesh vertex must also be a vertex of the supermesh, while the second states that every supermesh cell is completely contained within exactly one cell of either parent mesh; cf.~\cite{Farrell2009Supermesh}. The supermesh construction is not unique, as any conforming refinement of a supermesh is also a supermesh. Efficient algorithms for computing the supermesh are available, cf.~\cite{libsupermesh-tech-report}. Supermesh cells always lie within the intersection of a single pair of parent mesh cells and therefore the number of supermesh cells is proportional to the number of intersecting cells $K$.

Requiring that meshes involved in finite element computations are quasi-uniform is a standard working assumption in the literature, see e.g.~\cite{brenner2007mathematical}. We now recall the definition of quasi-uniformity.
\begin{definition}[definition $4.4.13$ in \cite{brenner2007mathematical}]
	\label{def:quasi-uniformity}
	Let $D$ be a given domain and let $\{\T_{\hat{h}}\}$, be a family of tessellations of $D$ such that for $0<\hat{h}\leq 1$,
	\begin{align}
	\label{eq:outer_ball0}
	\max\{\diam e : e\in\T_{\hat{h}}\} \leq \hat{h}\diam D,
	\end{align}
	where $\diam D$ is the diameter of $D$. The family is said to be quasi-uniform if there exists $\hat{\rho} > 0$ (independent of $\hat{h}$) such that
	\begin{align}
	\label{eq:inner_ball0}
	\min\{\diam \underline{\B}_e: e \in \T_{\hat{h}}\}\geq \hat{\rho} \hat{h} \diam D,
	\end{align}
	for all $\hat{h}\in(0,1]$, where $\underline{\B}_e$ is the largest ball contained in $e$ such that $e$ is
	star-shaped with respect to $\underline{\B}_e$ (cf.~definition 4.2.2 in \cite{brenner2007mathematical}).
\end{definition}

To simplify the exposition of what follows, it is more convenient to use the following property of quasi-uniform tessellations.
\begin{lemma}
	\label{lemma:quasi-uniformity}
	Let $D$ be a bounded domain and let $\{\T_{\hat{h}}\}$ be a quasi-uniform family of tessellations. Then there exist $h$ and $\rho\in(0,1]$ with $0<h\leq c_d = \sqrt{2d/(d+1)}$ such that
	\begin{align}
	\label{eq:outer_ball}
	\max\{\diam \overline{\B}_e : e\in\T_{\hat{h}}\} \leq h\diam D,
	\end{align}
	where $\overline{\B}_e$ is the smallest ball containing $e$ and
	\begin{align}
	\label{eq:inner_ball}
	\min\{\diam \underline{\B}_e: e \in \T_{\hat{h}}\}\geq \rho h \diam D,
	\end{align}
	for all $h\in(0,c_d]$, where $\underline{\B}_e$ is the largest ball contained in $e$ such that $e$ is star-shaped with respect to $\underline{\B}_e$.
\end{lemma}

\begin{proof}
	Jung's theorem (see \cite{Jung1901,Jung1910}) states that for any compact set $e\subset \R^d$,
	\begin{align}
	\diam \overline{\B}_e \leq \sqrt{\frac{2d}{d+1}} \diam e = c_d\diam e.
	\end{align}
	Therefore
	\begin{align}
	\label{eq:step1}
	\max\{\diam \overline{\B}_e : e\in\T_{\hat{h}}\} &\leq c_d \max\{\diam e : e\in\T_{\hat{h}}\} \leq c_d\hat{h}\diam D,
	\end{align}
	where we have used \eqref{eq:outer_ball0} in the last step.
	Equation \eqref{eq:inner_ball0} also gives us,
	\begin{align}
	\label{eq:step2}
	\min\{\diam \underline{\B}_e: e \in \T_{\hat{h}}\} &\geq \hat{\rho} \hat{h} \diam D.
	\end{align}
	Equations \eqref{eq:step1} and \eqref{eq:step2} are the same as \eqref{eq:outer_ball} and \eqref{eq:inner_ball} respectively after setting $h=c_d\hat{h}$ and $\rho = \hat{\rho}/c_d$. Note that since necessarily
	\begin{align}
	\label{eq:step3}
	\min\{\diam \underline{\B}_e: e \in \T_{\hat{h}}\} \leq \max\{\diam \overline{\B}_e : e\in\T_{\hat{h}}\},
	\end{align}
	and by combining \eqref{eq:step3} with \eqref{eq:outer_ball} and \eqref{eq:inner_ball} it is clear that $\rho \leq 1$.
\end{proof}
It will be more convenient in the sequel to index a family of tesselations by $h$ instead of $\hat{h}$.

In what follows we also need two auxiliary lemmas. The first states that the constants $h$ and $\rho$ appearing in lemma \ref{lemma:quasi-uniformity} also provide a lower and upper bound for the number of cells of a quasi-uniform mesh.
\begin{lemma}
	\label{lemma:quasi-uniform_lower_upper_bounds}
	Let $D_h$ be a quasi-uniform tessellation of a bounded domain $D\subset\R^d$ with $n_D\geq 1$ cells and let $c_D = 2^{d}|D|/(c_\pi(d)\normalfont{\diam}(D)^d)$ with $c_\pi = 2$ in 1D, $c_\pi = \pi$ in 2D and $c_\pi = 4\pi/3$ in 3D, then
	\begin{align}
	c_D h^{-d} \leq n_D \leq  c_D \rho^{-d} h^{-d},
	\end{align}
	where $h$ and $\rho$ are the constants appearing in lemma \ref{lemma:quasi-uniformity}.
\end{lemma}
\begin{proof}
	Let $e_i\in D_h$ for $i=1,\dots, n_D$ be the cells of $D_h$. We compute a lower bound for $n_D$ by noting that the measure of each cell is less than or equal to the measure of the smallest ball containing it, which gives
	\begin{align}
	|D| = \sum\limits_i|e_i| \leq c_\pi2^{-d}\sum\limits_i\diam(\overline{\B}_{e_i})^d \leq c_\pi2^{-d}n_D h^d\diam(D)^d,
	\end{align}
	where we have used equation \eqref{eq:outer_ball} in the last step. The lower bound is obtained by solving for $n_D$. Similarly we obtain an upper bound by noting that the measure of each cell is larger than the measure of any ball it contains. This gives,
	\begin{align}
	|D| = \sum\limits_i|e_i| \geq c_\pi2^{-d}\sum\limits_i\diam(\underline{\B}_{e_i})^d \geq c_\pi2^{-d}n_D \rho^d h^d\diam(D)^d,
	\end{align}
	where we used equation \eqref{eq:inner_ball} in the last step. Solving for $n_D$ yields the upper bound.
\end{proof}
For any $G\subset\R^d$, we use the notation $G\subset\joinrel\subset\R^d$ to indicate that the closure of $G$, $\bar{G}$, is a compact subset of $\R^d$.
The second lemma we need provides an upper bound for the number of intersections between the cells of a mesh $D_h$ and any $G\subset\joinrel\subset\R^d$.
\begin{lemma}
	\label{lemma:intersections_bounded_domain}
	Let $D_h$ be a quasi-uniform tessellation of a bounded domain $D\subset\R^d$ with $n_D\geq 1$ cells. Let $G\subset\joinrel\subset\R^d$ and, for a fixed $\delta>0$, define its $\delta$-fattening $F_\delta(G)$ to be the set of all points in $\R^d$ with distance from $\bar{G}$ smaller or equal than $\delta$ (with the convention that $F_\delta(\emptyset)=\emptyset$ for all $\delta$). Let $\I(G,D_h)$ be the number of cells of $D_h$ that intersect with $G$, then
	\begin{align}
	\I(G,D_h) \leq \min\left(\rho^{-d}\frac{|F_{\tilde{h}}(G) \cap D|}{|D|},1\right)n_D
	\end{align}
	where $\tilde{h}=h\diam D$ and $h$ and $\rho$ are the constants appearing in lemma \ref{lemma:quasi-uniformity}.
\end{lemma}
\begin{proof}
	We first devise a crude criterion for excluding the possibility that a given $e\in D_h$ can intersect with $G$. Since the diameter of $e$ is less than $\tilde{h}$, if the minimum distance between $e$ and $G$ is greater than or equal to $\tilde{h}$, the cell cannot possibly intersect $G$. That is,
	\begin{align}
	|e \cap G| = 0,\quad\text{if}\quad \max\limits_{\bm{x}\in e}\text{dist}(\bm{x},G) > \tilde{h},
	\end{align}
	which implies that
	\begin{align}
	|e \cap G| = 0,\quad\text{if}\quad \max\limits_{\bm{x}\in e}\text{dist}(\bm{x},F_{\tilde{h}}(G)) > 0,
	\end{align}
	i.e.~all intersecting cells must entirely be contained in $F_{\tilde{h}}(G)$, and, more specifically, in $F_{\tilde{h}}(G) \cap D$ since $e\subseteq D$. Therefore, we have that
	\begin{align}
	\label{eq:lemma0}
	\I(G,D_h)\leq \P(D_h,F_{\tilde{h}}(G) \cap D),
	\end{align}
	where $\P(D_h,F_{\tilde{h}}(G) \cap D)$ is the number of cells of $D_h$ that can be packed within $F_{\tilde{h}}(G) \cap D$ without overlapping. In turn, we can bound
	\begin{align}
	\label{eq:lemma1}
	\P(D_h,F_{\tilde{h}}(G) \cap D) \leq \P(\B_{\rho\tilde{h}},F_{\tilde{h}}(G) \cap D),
	\end{align}
	where with abuse of notation we denote with $\P(\B_{\rho\tilde{h}},F_{\tilde{h}}(G) \cap D)$ the number of balls of diameter $\rho\tilde{h}$ that can be packed within $F_{\tilde{h}}(G) \cap D$ without overlapping. This bound holds since all cells of $D_h$ entirely contain a ball of this diameter by quasi-uniformity and lemma \ref{lemma:quasi-uniformity}. Finding the sharpest possible upper bound for the packing of balls within a domain is a classical and extremely difficult problem in geometry called the ball packing problem (see e.g.~\cite{toth1993} for a survey). A crude upper bound is given by
	\begin{align}
	\label{eq:lemma2}
	\P(\B_{\rho\tilde{h}},F_{\tilde{h}}(G) \cap D) \leq \frac{|F_{\tilde{h}}(G) \cap D|}{|\B_{\rho\tilde{h}}|},
	\end{align}
	since the the sum of the measures of all the packed non-overlapping balls cannot exceed the measure of the set into which they are packed.
	Now, we have that
	\begin{align}
	\label{eq:lemma3}
	|\B_{\rho\tilde{h}}|^{-1} = c_\pi(d)^{-1}2^d(\rho h \diam(D))^{-d}\leq |D|^{-1}\rho^{-d}n_D,
	\end{align}
	where $c_\pi(d)$ is the same constant as in lemma \ref{lemma:quasi-uniform_lower_upper_bounds}, which we used to obtain the upper bound. Combining equations \eqref{eq:lemma0}-\eqref{eq:lemma3} together, we obtain
	\begin{align}
	\I(G,D_h)\leq \rho^{-d} \frac{|F_{\tilde{h}}(G) \cap D|}{|D|}n_D.
	\end{align}
	Noting that necessarily $\I(G,D_h)$ cannot exceed $n_D$ concludes the proof.
\end{proof}

\section{Main result}

We now prove theorem \ref{th:quasi_uniform_supermeshing_main}. To the authors' knowledge, this result is new. 

\begin{proof}
	Since $A_h$ and $B_h$ are quasi-uniform, we have by lemma \ref{lemma:quasi-uniformity},
	\begin{align}
	\max\{\diam \overline{\B}_e : e\in A_h\} &\leq h_A\diam A = \tilde{h}_A, \label{eqn:tildehA}\\
	\max\{\diam \overline{\B}_e : e\in B_h\} &\leq h_B\diam B = \tilde{h}_B, \label{eqn:tildehB}\\
	\min\{\diam \underline{\B}_e: e \in A_h\} &\geq \rho_A h_A \diam A = \rho_A\tilde{h}_A,\\
	\min\{\diam \underline{\B}_e: e \in B_h\} &\geq \rho_B h_B \diam B = \rho_B\tilde{h}_B,
	\end{align}
    for some $h_A,h_B\in(0,c_d]$, $\rho_A,\rho_B\in(0,1]$ independent of $h_A,\ h_B$.
    We have that
    \begin{align}
    \label{eq:n_nin_nout}
    n = n^{\text{in}} + n^{\text{out}},
    \end{align}
    where $n^{\text{in}}$ is the number of supermesh cells in $A \cap B$ and $n^{\text{out}}$ is the number of supermesh cells in $(A \cup B)\setminus (A \cap B)$.
    
    We will first provide a bound for $n_{\text{out}}$. We have that
    \begin{align}
    	n^{\text{out}} \leq \bar{c}(d)(\I(A\setminus B,A_h) + \I(B \setminus A, B_h))
    \end{align}
    where $\bar{c}(d)>0$ is the maximal number of simplices that the intersection between two cells of $A_h$ and $B_h$ is triangulated into; typical values for the intersection of convex polygons are given in \cite{Farrell2011Supermesh}. In this case we have $\bar{c}(1)=1$, $\bar{c}(2)=4$, and $\bar{c}(3)=45$. We can now apply lemma \ref{lemma:intersections_bounded_domain} to obtain
    \begin{align}
    \label{eq:nout}
    n^{\text{out}} &\leq \bar{c}(d)\left(\rho^{-d}_A\frac{|F_{\tilde{h}_A}(A\setminus B) \cap A|}{|A|}n_A + \rho^{-d}_B\frac{|F_{\tilde{h}_B}(B\setminus A) \cap B|}{|B|}n_B\right)\\
    &\leq \bar{c}(d)\max(\rho_A^{-d},\rho_B^{-d})(n_A+n_B)=c^{\text{out}}(d,A,B)(n_A+n_B).\notag
    \end{align}
    Note that from the first inequality we have that $n^{\text{out}}=0$ if ${A = B}$.

    We now derive an upper bound for $n^{\text{in}}$. Let $\I(A_h,B_h)$ be the number of intersecting cell pairs between $A_h$ and $B_h$. We have that
    \begin{align}
        \label{eq:0}
        n^{\text{in}} \leq \bar{c}(d)\I(A_h,B_h).
    \end{align}
    For a given cell $e_i^A\in A_h$, let $\I(e_i^A, B_h)$ be the number of cells of $B_h$ that intersect with $e_i^A$. We then have that
	\begin{align}
	\label{eq:1}
	\I(A_h, B_h) = \sum\limits_{i=1}^{n_A}\I(e_i^A,B_h).
	\end{align}
	Applying lemma \ref{lemma:intersections_bounded_domain} we obtain
	\begin{align}
	\label{eq:temp_bound0}
	\I(A_h, B_h) \leq \rho_B^{-d}|B|^{-1}n_B\sum\limits_{i=1}^{n_A}|F_{\tilde{h}_B}(e_i^A) \cap B|,
	\end{align}
	where each term in the sum can be bounded by
	\begin{align}
		|F_{\tilde{h}_B}(e_i^A) \cap B|\leq |F_{\tilde{h}_B}(\B_{\tilde{h}_A})|=|\B_{\tilde{h}_A+2\tilde{h}_B}|,\quad i=1,\dots,n_A,
	\end{align}
	where we used the fact that all cells $e_i^A \in A_h$ can be entirely contained within a ball of diameter $\tilde{h}_A$. We then have that
	\begin{align}
	\sum\limits_{i=1}^{n_A}|F_{\tilde{h}_B}(e_i^A) \cap B|&\leq n_A|\B_{\tilde{h}_A+2\tilde{h}_B}|\notag\\
	&=n_A c_\pi(d)2^{-d}(\tilde{h}_A+2\tilde{h}_B)^d \label{eq:temp_bound1}\\
	&= n_A(|A|^{1/d}c_A^{-1/d}h_A + 2|B|^{1/d}c_B^{-1/d}h_B)^d\notag
	\end{align}
	where $c_\pi(d)$, $c_A$ and $c_B$ are as in lemma \ref{lemma:quasi-uniform_lower_upper_bounds}. Note that we have removed the tildes since the diameter terms are included in $c_A$ and $c_B$. For instance,
	\begin{align}
	c_\pi(d)^{1/d}\tilde{h}_A/2 = \frac{c_\pi(d)^{1/d}h_A\diam(A)|A|^{1/d}}{2|A|^{1/d}} = |A|^{1/d}c_A^{-1/d}h_A.
	\end{align}
	Lemma \ref{lemma:quasi-uniform_lower_upper_bounds} yields the upper bounds
	\begin{align}
	c_A^{-1/d}h_A\leq \rho_A^{-1}n_A^{-1/d},\quad c_B^{-1/d}h_B\leq \rho_B^{-1}n_B^{-1/d}
	\end{align}
	Using these in \eqref{eq:temp_bound1} and combining it with \eqref{eq:temp_bound0} gives
	\begin{align}
	\label{eq:temp_bound2}
	\I(A_h, B_h) \leq \rho_B^{-d}|B|^{-1}n_An_B(|A|^{1/d}\rho_A^{-1}n_A^{-1/d} + 2|B|^{1/d}\rho_B^{-1}n_B^{-1/d})^d\\
	\leq \min(\lambda_{AB}\rho_A\rho_B,\rho_B^2/2)^{-d}(n_A^{1/d} + n_B^{1/d})^d,\quad\text{where}\quad\lambda_{AB}=|A|/|B|.\notag
	\end{align}
	Combining\eqref{eq:0} and \eqref{eq:temp_bound2} we have
	\begin{align}
	n^{\text{in}} &\leq \bar{c}(d)\min(\lambda_{AB}\rho_A\rho_B,\rho_B^2/2)^{-d}(n_A^{1/d} + n_B^{1/d})^d\notag\\
	&\leq c^{\text{in}}(d,A,B)(n_A + n_B),
	\label{eq:nin}
	\end{align}
	where $c^{\text{in}}=2^{d-1}\bar{c}(d)\min(\lambda_{AB}\rho_A\rho_B,\rho_B^2/2)^{-d}$ since by Jensen's inequality,
	\begin{align}
        \left(\frac{x + y}{2}\right)^d \leq \frac{x^d+y^d}{2},\quad\forall x,y \geq 0,\ d \in\mathbb{N}^+,
	\end{align}
	due to convexity of the function $f(t)=t^d$ for all $t\geq0$, $d\in\mathbb{N}^+$.
    Combining \eqref{eq:n_nin_nout} with \eqref{eq:nin} and \eqref{eq:nout} yields the claim with $C(d,A,B)=\max(c^{\text{in}},c^{\text{out}})$. Note that $C(d, A, B)$ cannot be smaller than $1/2$. In fact, we have that
    \begin{align}
        \max(n_A,n_B) \leq n \leq C(d, A, B) (n_A + n_B),
    \end{align}
    and therefore
    \begin{align}
        \frac{\max(n_A, n_B)}{n_A + n_B} \leq C(d, A, B),\ \text{with } n_A, n_B \geq 1.
    \end{align}
    The quantity on the left attains its minimum value $1/2$ when $n_A = n_B$.
\end{proof}

\section{Numerical results}

The constant in the bound derived in theorem \ref{th:quasi_uniform_supermeshing_main} is not sharp. However, a sharper constant can easily be estimated in practice. We now offer an example in the case in which $A$ and $B$ coincide and are equal to $D = (-0.5,0.5)^d$ in 2D and 3D.

We consider a hierarchy $\{D_h^\ell\}_{\ell=1}^{\ell = L}$ of non-nested unstructured grids of the domain $D$, with $L=9$ in 2D and $L=5$ in 3D. The grids have mesh size decreasing geometrically with $\ell$ so that $h_\ell \propto 2^{-\ell}$, where $h_\ell$ is the constant $h$ appearing in lemma \ref{lemma:quasi-uniformity} for the mesh $D_h^\ell$. Hierarchies of this type are commonly used within non-nested geometric multigrid methods for which a conservative transfer technique is required.

For each $\ell > 1$ we construct a supermesh $S_h^\ell$ between $D_h^\ell$ and $D_h^{\ell-1}$ using the libsupermesh library \cite{libsupermesh-tech-report} and we compute the ratio
\begin{align}
    R_\ell = \dfrac{n_{S^\ell}}{n_{\ell} + n_{\ell-1}},
\end{align}
where $n_{S^\ell}$, $n_\ell$ and $n_{\ell-1}$ are the number of cells of $S_h^\ell$, $D_h^\ell$ and $D_h^{\ell-1}$ respectively. The value of $R_\ell$ clearly provides an estimate for the constant $C(d,A,B)$ appearing in theorem \ref{th:quasi_uniform_supermeshing_main} in this specific setting. The results are shown in tables \ref{tab:0}a and \ref{tab:0}b together with some additional information about the meshes involved. In both 2D and 3D the value of $R_\ell$ appears to monotonically increase with $\ell$ and to plateau for large $\ell$ to a value which is approximately $4$ in 2D and $40$ in 3D. The observation that $R_\ell$ plateaus is expected from the bound given by theorem \ref{th:quasi_uniform_supermeshing_main}.

\begin{table}[h!]
	\parbox{0.49\linewidth}{
		\centering
		\begin{tabular}{cccc}
			\toprule
			$\ell$ & $h_\ell$      & $\rho_\ell$ & $R_\ell$ \\ \midrule
			$1$    & $0.5878$ & $0.35$ &  n/a     \\
			$2$    & $0.3159$ & $0.31$ &  $2.8$   \\
			$3$    & $0.1455$ & $0.26$ &  $3.2$   \\
			$4$    & $0.7777$ & $0.20$ &  $3.6$   \\
			$5$    & $0.0393$ & $0.21$ &  $3.7$   \\
			$6$    & $0.0205$ & $0.20$ &  $3.8$   \\
			$7$    & $0.0108$ & $0.18$ &  $3.9$   \\
			$8$    & $0.0054$ & $0.18$ &  $3.9$   \\
			$9$    & $0.0027$ & $0.17$ &  $3.9$   \\
			\bottomrule\vspace{-9pt}
		\end{tabular}\\
	(a)
	}
	\parbox{0.49\linewidth}{
		\vspace{-0.89cm}
		\centering
		\begin{tabular}{cccc}
			\toprule
			$\ell$ & $h_\ell$      & $\rho_\ell$ & $R_\ell$ \\ \midrule
			$1$    & $0.4330$ & $0.141$  &  n/a     \\
			$2$    & $0.2173$ & $0.056$  & $22$    \\
			$3$    & $0.1061$ & $0.052$ &  $33$    \\
			$4$    & $0.0534$ & $0.040$ &  $36$    \\
			$5$    & $0.0252$ & $0.035$ &  $38$    \\
			$6$    & $0.0123$ & $0.027$ &  $39$    \\
			$7$    & $0.0057$ & $0.027$ &  $39$    \\
			\bottomrule\vspace{-9pt}
		\end{tabular}\\
		(b)
	}
	\caption{Mesh hierarchies considered in 2D (a) and 3D (b) and the resulting ratio $R_\ell$ between the number of cells of the supermesh between $D_h^\ell$ and $D_h^{\ell-1}$ and the sum of the numbers of cells of the parent meshes. We indicate with $h_\ell$ and $\rho_\ell$ the constants appearing in lemma \ref{lemma:quasi-uniformity} relative to the mesh $D_h^\ell$. Note that the mesh size $h_\ell$ is roughly proportional to $2^{-\ell}$ and that $R_\ell$ appears to be plateauing as $\ell$ grows, as predicted by theorem \ref{th:quasi_uniform_supermeshing_main}.}
	\label{tab:0}
\end{table}
\vspace{-0.2cm}

%

\section{Conclusions}

In this note we have improved on the previously known bound for the number of cells of a supermesh constructed between two meshes. Under the natural assumption that the parent meshes are quasi-uniform, we have showed that the number of supermesh cells is linear in the number of cells of the parent meshes. Numerical experimentation confirms the theory. This result is important for the analysis of algorithms that rely on supermesh construction, such as conservative interpolation and non-nested multilevel Monte Carlo algorithms.

\section*{Acknowledgements}
This publication is based on work partially supported by the EPSRC Centre For Doctoral Training in Industrially Focused Mathematical Modelling (EP/L015803/1) in collaboration with Simula Research Laboratory,
and by EPSRC grant EP/R029423/1.

\printbibliography

\end{document}